\documentclass{article}

\usepackage{amsmath,amssymb,mathtools}
\newcommand{\Rbb}{\mathbb{R}}

\newcommand{\ufrak}{\mathfrak{u}}
\newcommand{\rfrak}{\mathfrak{r}}

\newcommand{\oderiv}[1]{ \frac{d^{#1}}{dt^{#1}} }
\newcommand{\Matern}{Mat\'ern}
\newcommand{\Ocal}{\mathcal{O}}
\newcommand{\Ical}{\mathcal{I}}

\usepackage{sansmath} 

\usepackage[final]{ProbNum25}

\usepackage[round]{natbib}

\usepackage[capitalise,nameinlink]{cleveref}

\usepackage{style_commands}

\usepackage{kantlipsum}

\probnumtitle{Adaptive Probabilistic ODE Solvers Without \\ Adaptive Memory Requirements}
\probnumauthors{%
\name{Nicholas Krämer}%
\affiliation{Technical University of Denmark, Kongens Lyngby, Denmark}
}

\probnumabstract{Despite substantial progress in recent years, probabilistic solvers with adaptive step sizes can still not solve memory-demanding differential equations---unless we care only about a single point in time (which is far too restrictive; we want the whole time series). Counterintuitively, the culprit is the adaptivity itself: Its unpredictable memory demands easily exceed our machine's capabilities, making our simulations fail unexpectedly and without warning. Still, dropping adaptivity would abandon years of progress, which can't be the answer. In this work, we solve this conundrum. We develop an adaptive probabilistic solver with fixed memory demands building on recent developments in robust state estimation. Switching to our method (i) eliminates memory issues for long time series, (ii) accelerates simulations by orders of magnitude through unlocking just-in-time compilation, and (iii) makes adaptive probabilistic solvers compatible with scientific computing in JAX.
    \par\bigskip
    \begin{center}
    \begin{tabular}{ll}
    \textbf{JAX library:} 
        &\LibraryInstall{}
        \\
    \textbf{Experiment code:} 
        &\LibraryExperiments{}
    \end{tabular}
    \end{center}
}

\begin{document}

\section{Introduction: Adaptive Probabilistic Solvers vs. Limited Memory}

Probabilistic numerical methods solve problems from numerical computing by manipulating random variables \citep{hennig2022probabilistic}.
For example, conditioning a stochastic process $u(t)$ on the constraints $\left\{u'(t_n) = u(t_n)\,\right\}_{n=1}^N$ and the initial condition $u(0) = u_0$ approximates the solution of $u'(t) = u(t)$, $u(0) = u_0$ \citep{cockayne2019bayesian}.
Unlike more traditional techniques, these new algorithms emphasise uncertainty quantification and the interplay between simulators and observational data.
For instance, the works by \citet{kersting2020differentiable,tronarp2022fenrir,schmidt2021probabilistic,beck2024diffusion,oesterle2022probabilistic,oates2019bayesian,lahr2024probabilistic} show the utility of this new paradigm for uncertainty quantification of differential equations.
Unfortunately, the probabilistic solvers' usefulness is limited by memory, as follows:

Since solving nonlinear differential equations requires time discretisation, users must provide implementations with either (i) a grid; or (ii) an error tolerance, according to which the solver selects the grid automatically and adaptively.
This adaptive step-size selection has been one of the great improvements to the practicality of probabilistic solvers \citep{schober2019probabilistic,bosch2021calibrated}.
The reason is that it drastically lessens the resolution needed for solving challenging problems.
For example, for \citeauthor{van1920theory}'s \citeyearpar{van1920theory} system,
\begin{align}
\label{equation-van-der-pol}
    u''(t) = 10^3 (u'(t) (1 - u(t)^2) - u(t)),
\end{align}
adaptivity reduces the number of steps, thus the solver's memory footprint and general ``amount of work'', from almost 750,000 to under 3,000 grid points, while achieving similar accuracy (\Cref{figure-van-der-pol-steps}).
\begin{figure}[t]
    \begin{center}
        \includegraphics[width=0.95\linewidth]{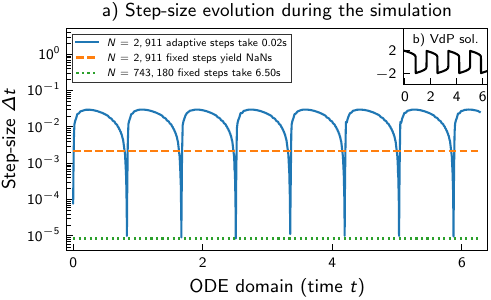}
    \end{center}
    \caption{
    \begin{sansmath}
        {Adaptive step-size selection minimises the amount of work for solving (stiff) differential equations (a)}.
        Van-der-Pol system (solution in b), tolerance $0.001$.
        One adaptive versus two fixed-grid solvers: one matches the adaptive solvers' number of steps (and fails), and one matches its accuracy (but requires over $250\times$ more memory \& runtime).
        Detailed setup: \Cref{appendix-setup-figure-van-der-pol-steps}.
        \end{sansmath}
    }
    \label{figure-van-der-pol-steps}
\end{figure}
Furthermore, this $250\times$ memory improvement comes with a proportional runtime reduction, but we mainly care about memory in this work because many simulations, for instance, those in \Cref{section-experiment-overcoming-memory-limitations}, are memory-limited and not runtime-limited.
Our work shows that one can decrease memory requirements without increasing runtime. In practice, JAX-code actually becomes faster since just-in-time (JIT) compilation is now possible; more on this later.

In other words, adaptive step-size selection is critical for decreasing the solvers' memory demands without losing accuracy.
\emph{However, it doesn't reduce them enough:}
The issue is that adaptive solvers represent the solution with all points used for the computation, regardless of whether or not only a tiny subset of these points is relevant to what we plan to do with the solution.
For example, if we were to estimate an unknown initial condition of the van-der-Pol system (\Cref{equation-van-der-pol}) from a single observation in the middle of the time domain, we would still have to store all $\approx 3,000$ states used for the simulation.
This excessive storage is problematic:

First, it consumes more memory than necessary, potentially too much for the simulation to be feasible.
    \textit{Our work enables high-precision simulations of challenging systems on a small laptop, which exceeds prior works' memory capabilities.}
    \Cref{section-experiment-overcoming-memory-limitations} shows an example. 
    
Second, not knowing the memory footprint of the simulation complicates using programming frameworks that must know the exact memory demands of a program before compiling/running it, for example, JAX \citep{jax2018github}.
    \textit{Our work enables JIT-compiling JAX-code for adaptive probabilistic solvers.} 
    \Cref{section-experiment-rigid-body} shows how this drastically improves performance.

To illustrate the second problem: until the present work, existing JAX-based implementations, for instance this\footnote{\texttt{https://github.com/pnkraemer/tornadox}} one or this\footnote{\texttt{https://github.com/mlysy/rodeo}} one, either use fixed steps, adaptively simulate only the terminal value (ignoring the interior of the time domain) or can't compile the solution routines.
Our work doesn't suffer from any of these problems.
An open-source JAX implementation of our method is available under this\footnote{\LibraryLink{}} link; install it via
\LibraryInstall{}.
The lack of JIT'able adaptive solvers in JAX was the original motivation for the research presented in here.

\section{Problem Statement}
\label{section-problem-statement}

To summarise, the present work solves a practical problem with probabilistic differential equations solvers:
\begin{problem}
{We aim to solve a differential equation adaptively, to a user-specified tolerance, and with memory requirements that are fixed, independent of the time-stepping, and known \underline{before} the simulation.}
\end{problem}
We call this process of adaptively generating the solution at a pre-specified set of target points ``adaptive target simulation,'' as opposed to what we call ``adaptive simulation:'' returning the adaptive solution at all locations used for the approximation.
We could combine adaptive simulation with interpolation to achieve adaptive target simulation. 
Still, this combination would inherit all the memory-related problems from adaptive simulation, and our goal is to avoid them.
Interpolation is thus not the correct answer;
and using what often works for traditional solvers isn't either, as follows.

Adaptive target simulation for non-probabilistic solvers, like Runge--Kutta methods, usually reduces to a decision of whether or not to save the current step.
Past and future increments are rarely affected.
For example, passing a non-empty \textit{t\_eval} (the target locations) to SciPy's \textit{solve\_ivp} affects only very few lines of code: storing an interpolation instead of the time-step.\footnote{%
\texttt{https://docs.scipy.org/doc/scipy/reference/
generated/scipy.integrate.solve\_ivp.html}}
However, the case for probabilistic solvers is more complicated because they operate through a forward- and a backward-pass, where the backward-pass depends on all intermediate results of the forward pass \citep[e.g.][]{kramer2024implementing}.
Because of this dependence, all intermediate states have to be saved if we use existing formulations of adaptive solvers.
For reference, the example in \Cref{figure-van-der-pol-steps} uses approximately 3,000 adaptive steps, and all of those need to be saved to parametrise the posterior distribution.
However, storing even 3,000 locations may not be feasible for high-dimensional problems like discretised partial differential equations (\Cref{section-experiment-overcoming-memory-limitations}).

In this work, we show that storing all intermediate steps is unnecessary.
We modify the solver recursion so that the backward pass only depends on the solution's values at all target locations instead of those at compute locations.
More technically, we extend recent advances in numerically robust fixed-point smoothing \citep{kramer2024numerically} to simulating differential equations and then combine them with adaptive time-stepping.
The result is a previously unknown implementation of adaptive probabilistic ODE solvers. 
However, existing codes need not be trashed:
We can update them by tracking a single additional variable at each stage of the forward pass.
The rest of this paper details how.

\section{Method: Adaptive Target Simulation via Fixed-Point Smoothing}
\label{section-method}

\subsection{Background on Probabilistic Solvers}

This work discusses solving ordinary-differential-equation-based initial value problems (ODEs).
For example, the SIR model \citep{kermack1927contribution} is an ODE that describes a disease outbreak, starting at initial case counts and evolving according to specific nonlinear dynamics.
More specifically, let $f: \Rbb^d \rightarrow \Rbb^d$ be a known, potentially nonlinear vector field and $u_0 \in \Rbb^d$ be a known vector.
Consider the ODE
\begin{align}
    \label{equation-ode}
    \oderiv{} u(t) = f(u(t)), \quad t \in [0, 1], \quad u(0) = u_0,
\end{align}
and assume that it admits a unique, sufficiently regular solution.
The precise form of \Cref{equation-ode} does not matter much; second derivatives or time-varying dynamics only need minimal changes, and the general template continues to apply \citep{bosch2022pick}.
In other words, we use \Cref{equation-ode} to explain the algorithm but it applies to general ODEs via \citet{bosch2022pick}.
For instance, the simulation in \Cref{figure-van-der-pol-steps} involves second derivatives.

Numerical solution of ODEs requires approximation because the vector field $f$ is nonlinear.
Approximation requires time-discretisation:
Introduce the point sets
\begin{subequations}
    \begin{align}
        \text{``Compute grid'':} &
                                 &
        t_{0:N}
                                 & \coloneqq
        \{t_0, ..., t_N\},
        \\
        \text{``Target grid'':}  &
                                 &
        s_{0:M}
                                 & \coloneqq
        \{s_0, ..., s_M\}.
    \end{align}
\end{subequations}
The target grid $s_{0:M}$ is specified by the user, and the compute grid $t_{0:N}$ selected by the algorithm.
No technical relation exists between $t_{0:N}$ and $s_{0:M}$, but we optimise our method for $M \ll N$.
Without a loss of generality, assume both grids' endpoints coincide with the domain's boundary, $t_0 = s_0 = 0$ and $t_N = s_M = 1$. If not, augment the point sets.
Treating the compute- and the target-grid differently is central to this work.

Probabilistic numerical solvers condition a prior distribution (usually Gaussian) on satisfying the ODE on the compute-grid $t_{0:N}$.
Let $u$ be a Gaussian process that admits $L$ or more derivatives (in the mean-square sense, like it's typical in the Gaussian process literature \citep{williams2006gaussian}), and assume that $L$ is at least as large as the highest derivative in the ODE (\Cref{equation-ode}).
For example, ODEs involving first derivatives need $L \geq 1$ which is satisfied by, for example, a \Matern{}$\left(\tfrac{3}{2}\right)$ or a \Matern{}$\left(\tfrac{5}{2}\right)$ process.
For ODEs involving second derivatives, the \Matern{}$\left(\tfrac{3}{2}\right)$ would not be regular enough.
$L$ can be thought of as the order of the solver; the higher $L$ is (with everything else being equal), the more accurate the approximation, but also, the more expensive and ill-conditioned the solution process becomes \citep{kramer2024stable}.

Since differentiation is linear,
\begin{align}
    \ufrak \coloneqq \left[u, \oderiv{} u, \oderiv{2} u, ... \oderiv{L}u\right]
\end{align}
is a Gaussian process.
Abbreviate the ODE residual
\begin{align}
    \rfrak(t) \coloneqq \oderiv{} u(t) - f(u(t)).
\end{align}
This $\rfrak$ is a nonlinear function of $\ufrak$; recall that $\ufrak$ contains $\oderiv{} u$.
A globally small $\rfrak$ implies that a given $u$ approximates the ODE solution accurately.
Adopting \citeauthor{kramer2024implementing}'s \citeyearpar{kramer2024implementing} terminology, we call any approximation of %
\begin{align}
    \label{equation-probabilistic-ode-solution}
    p(\ufrak(s_{0:M}) \mid \rfrak(t_{0:N})=0, u_0)
\end{align}
a ``probabilistic ODE solution''.
Any algorithm that computes such a probabilistic solution shall be a ``probabilistic solver''.
For the rest of this paper, we drop the ``probabilistic'' and only write ``solver/solution'' because all methods will be of this type.
Recall how $t_{0:N}$ is selected adaptively during the forward pass, whereas the user chooses $s_{0:M}$.
\textit{Approximating the solution in \Cref{equation-probabilistic-ode-solution} with memory-complexity dictated by $s_{0:M}$ and not by $t_{0:N}$ has been previously unknown and is the main contribution of this paper.}
More technically, we now introduce a method with $\Ocal(N+M)$ runtime and $\Ocal(M)$ memory, as opposed to the $\Ocal(N+M)$ runtime and $\Ocal(N+M)$ memory of existing approaches.

\subsection{The Method}
A sequential implementation of the ODE solver needs a Markovian prior; concretely, it needs \Cref{assumption-gauss-markov}:
\begin{assumption}
\label{assumption-gauss-markov}
Assume that $\ufrak$ has the Markov property and that the transition density is
\begin{align}
p(\ufrak(t + \Delta t) \mid \ufrak(t)) = N(\Phi(\Delta t) \ufrak(t), \Sigma(\Delta t))
\end{align}
with known matrices $\Phi(\Delta t)$ and $\Sigma(\Delta t)$.
\end{assumption}
\Cref{assumption-gauss-markov} is not restrictive if we compare with earlier work on selecting priors \citep{schober2019probabilistic,bosch2024probabilistic,kramer2022million}. 
All Gaussian processes that solve linear, time-invariant stochastic differential equations, for example, \Matern{} or integrated Wiener processes, satisfy it \citep{sarkka2019applied}.

With \Cref{assumption-gauss-markov}, proceed in two steps:
First, \Cref{section-two-targets} derives an $\Ocal(1)$-memory solver assuming two target points, $s_{0:M} = \{0, 1\}$, instead of arbitrarily many.
The $\Ocal(1)$ memory is critical, and novel in this context.
The resulting procedure will serve as the foundation for the general case.
Second, \Cref{section-many-targets} extends the $\Ocal(1)$-memory code for two to an $\Ocal(M)$-memory version for many targets by calling the two-target code repeatedly and in a specific way.
The resulting method then solves the problem stated in \Cref{section-problem-statement}.

\subsubsection{Two Target Points}
\label{section-two-targets}
For now, assume $s_{0:M} = \{0, 1\}$ and recall that we assume that the compute grid contains both $t_0 = 0$ and $t_N = 1$.
A sequential, in-place algorithm for estimating
\begin{align}
\label{equation-ode-solution-two-targets}
\begin{split}
&p(\ufrak (0), \ufrak(1) \mid \rfrak(t_{0:N}), u_0)
\\
&
=
p(\ufrak (0) \mid \ufrak(1), \rfrak(t_{0:N}), u_0)
p(\ufrak (1) \mid \rfrak(t_{0:N}), u_0).
\end{split}
\end{align}
arises as follows.
Marginalising over $\ufrak(t_{N-1})$ and using Markovianity turns the first term in \Cref{equation-ode-solution-two-targets} into
\begin{subequations}
\begin{align}
&p(\ufrak(0) \mid \ufrak(1), \rfrak(t_{0:N}), u_0)
\notag
\\
&
=
\int p(\ufrak (0), \ufrak(t_{N-1}) \mid \ufrak(1), \rfrak(t_{0:N}), u_0) \diff \ufrak(t_{N-1})
\\
&
=
\int 
    \pi_0(\ufrak(t_{N-1})) \pi_1(\ufrak(t_{N-1}))
\diff \ufrak(t_{N-1}),
\label{equation-detour-via-tn-temporary-variables}
\end{align}
\end{subequations}
where \Cref{equation-detour-via-tn-temporary-variables} uses the temporary variables
\begin{subequations}
\begin{align}
\pi_0(\ufrak(t_{N-1}))
&\coloneqq
p(\ufrak(0) \mid \ufrak(t_{N-1}), \rfrak(t_{0:N-1}), u_0)
\\
\pi_1(\ufrak(t_{N-1}))
&\coloneqq
p(\ufrak(t_{N-1}) \mid \ufrak(1),  \rfrak(t_{0:N-1}), u_0)
\end{align}
\end{subequations}
to simplify the notation.
By applying the same idea to $\pi_0(\ufrak(t_{N-1}))$, namely stepping through $\ufrak(t_{N-2})$ and using Markovianity, and repeating the approach for all remaining points in $t_{1:N-1}$ (backwards over the indices), we obtain the marginalised, backwards factorisation
\begin{align}
\label{equation-markov-factorisation-conditional}
&p(\ufrak(0) \mid \ufrak(1), \rfrak(t_{0:N}), u_0)
\\
&=
\int
\left(
\prod_{n=0}^{N-1}
    p(\ufrak(t_{n}) \mid \ufrak(t_{n+1}), \rfrak(t_{0:n}), u_0)
\right)
\diff \ufrak(t_{1:N-1})
.
\notag
\end{align}
We make three critical observations:

First, the parametrisation of each conditional factor $p(\ufrak(t_n) \mid \ufrak(t_{n+1}), \rfrak(t_{0:n}), u_0) $ in \Cref{equation-markov-factorisation-conditional}  
 depends on that of $p(\ufrak(t_n) \mid \rfrak(t_{0:n}), u_0)$ and the transition density of the prior. The latter is known due to \Cref{assumption-gauss-markov}.
In other words, each factor in \Cref{equation-markov-factorisation-conditional} can be parametrised as a byproduct of the $n$-th step of the forward pass of the ODE solver.
This insight is critical but not new \citep[e.g.,][Equation 3.15]{kramer2024implementing}.

Second, we can compute \Cref{equation-markov-factorisation-conditional} in $\Ocal(1)$ memory and $\Ocal(N)$ runtime as long as we marginalise in the correct order, which is the following:
for $n=0$, initialise 
\begin{align}
\label{equation-marginal-initialisation}
p(\ufrak(0) \mid \ufrak(0), \rfrak(t_0), u_0) = \delta(\ufrak(0)).
\end{align}
Store the Dirac delta $\delta$ as a Gaussian with zero covariance to request the correct amount of storage for later steps. 
To simplify the upcoming notation, define the following transitions for $k=0, ..., n-1$,
\begin{align}
\rho(\ufrak(t_{k}), \ufrak(t_n))
\coloneqq
p(\ufrak(t_k) \mid \ufrak(t_{n}), \rfrak(t_{0:n}), u_0).
\end{align}
The start in \Cref{equation-marginal-initialisation} parametrised $\rho(\ufrak(0), \ufrak(0))$.
After initialising, iterate for $n=0, ..., N-1$,
\begin{align}
\begin{split}
\label{equation-merge-conditionals}
&
\rho(\ufrak(0), \ufrak(t_{n+1}))
\\
&~=
\int 
\rho(\ufrak(0), \ufrak(t_n))
\rho(\ufrak(t_{n}), \ufrak(t_{n+1}))
\diff \ufrak(t_n).
\end{split}
\end{align}
Iterating \Cref{equation-merge-conditionals} consumes $\Ocal(1)$ memory because storage for the parameters of $\rho$ can be reused.
This in-place behaviour is critical to an $\Ocal(1)$ memory code and novel for probabilistic ODE solvers.
\Cref{equation-merge-conditionals} holds generally; \Cref{assumption-gauss-markov} implies that it is straightforward to evaluate (discussed next).

Third, the marginalisation in \Cref{equation-merge-conditionals} is possible in closed form because $\rho$ is always an affine Gaussian transformation under \Cref{assumption-gauss-markov}.
\Cref{appendix-section-robust-marginalisation} discusses a numerically robust implementation, mimicking \citeauthor{kramer2024numerically}'s \citeyearpar{kramer2024numerically} fixed-point smoother. We revisit the connection to fixed-point smoothing below.

\Cref{algorithm-two-target-points} summarises the procedure.
\begin{algorithm*}[t]
\caption{%
    \begin{sansmath}
    {Adaptive simulation in $\Ocal(1)$ memory.}
    ``Predict'' \& ``step'' as usual.
    Marginalisation: \Cref{appendix-section-robust-marginalisation}.
    \end{sansmath}
}
\label{algorithm-two-target-points}
\begin{algorithmic}[1]
\Require $\left[p(\ufrak(a)), a\right]$ to start interpolation, $\left[ p(\ufrak(t)), t\right]$ to start time-stepping, $p(\ufrak(a) \mid \ufrak(t))$ (which could be the identity), initial step-size $\Delta t$, right-hand side boundary $b \in (a, 1]$. 
\Ensure $a \leq t$ \Comment{Both $t \leq b$ and $b\leq t$ are possible.}
\If{$b \in (a, t)$}
    \Comment{Early exit in case no time-stepping is necessary}
    \State $p(\ufrak(b)), p(\ufrak(a) \mid \ufrak(b))  = \text{predict}(p(\ufrak(a)), b-a)$
    \Comment{See \citeauthor{kramer2024implementing}'s \citeyearpar{kramer2024implementing} summary}
    \State $ \_\_, p(\ufrak(b) \mid \ufrak(t)) = \text{predict}(p(\ufrak(b)), t-b)$
    \Comment{Ignore one output}
\Else
    \Comment{$b$ must be at least as large as $t$, so we start time-stepping}
    \State $t_\text{prev} \leftarrow a$
    \Comment{Carry two intervals: $(t_\text{prev}, t]$ \& $(t, t_\text{next}]$}
    \State $p(\ufrak(a) \mid \ufrak(t_\text{prev}) \leftarrow N(\ufrak(t_\text{prev}), 0)$
    \While{$t < b$}
        \State $p(\ufrak(t_\text{next})), p(\ufrak(t) \mid \ufrak(t_\text{next})), \Delta t = \text{step}(p(\ufrak(t)), \Delta t)$
        \Comment{See \citeauthor{kramer2024implementing}'s \citeyearpar{kramer2024implementing} summary}
        \State $p(\ufrak(a) \mid \ufrak(t)) \leftarrow \int p(\ufrak(a) \mid \ufrak(t_\text{prev})) p(\ufrak(t_\text{prev}) \mid \ufrak(t))\diff \ufrak(t_\text{prev})$
        \Comment{new!}
        \State $(t_\text{prev}, t) \leftarrow (t, t_\text{next})$
        \EndWhile
    \State $p(\ufrak(b)), p(\ufrak(t_\text{prev}) \mid \ufrak(b)) = \text{predict}(p(\ufrak(t_\text{prev})), b-t_\text{prev})$
    \State $p(\ufrak(a) \mid \ufrak(b)) \leftarrow \int p(\ufrak(a) \mid \ufrak(t_\text{prev})) p(\ufrak(t_\text{prev}) \mid \ufrak(b))\diff \ufrak(t_\text{prev})$
    \Comment{new!}
    \State $\_\_, p(\ufrak(b) \mid \ufrak(t)) = \text{predict}(p(\ufrak(b)), t-b)$
\EndIf
\State
\Return $\left\{p(\ufrak(a) \mid \ufrak(b)), p(\ufrak(b)) \right\}$, $\left\{p(\ufrak(b) \mid \ufrak(t)), p(\ufrak(t))\right\}$, $\Delta t$
\end{algorithmic}
\end{algorithm*}
Note how it slightly deviates from the above derivation:
It solves from $a \in [0, 1]$ to $b \in (a, 1]$ instead of from $0$ to $1$.
It also contains an initial block that checks whether time-stepping is necessary.
Both are important when we solve for many target points instead of two, which will become clear in \Cref{section-many-targets}.
``Predict'' and ``step'', including generating $\Delta t$, match \citet{bosch2021calibrated,kramer2024stable}. 
Thus, we omit the details and refer to \citeauthor{kramer2024implementing}'s \citeyearpar{kramer2024implementing} tutorial instead. 
The marginalisation of the conditionals is new, and using it for constant-memory ODE solvers our contribution. 
\Cref{algorithm-two-target-points} distils creating $\Ocal(1)$ memory ODE solvers into, essentially, a single line of code.
This distillation is helpful since we can update published solvers without rewriting existing, sophisticated solution routines.

\Cref{algorithm-two-target-points} generalises \citeauthor{kramer2024numerically}'s \citeyearpar{kramer2024numerically} fixed-point smoother:
Namely, if all compute-grid-points $t_{0:N}$ were known in advance, the ODE were affine, and if we followed \Cref{algorithm-two-target-points} up with one last marginalisation that yields $p(\ufrak(0) \mid \rfrak(t_{0:N}), u_0)$, we would recover  \citeauthor{kramer2024numerically}'s \citeyearpar{kramer2024numerically}  fixed-point smoother.
However, our method is more general: our grid points need not be known in advance, the ODE is nonlinear, and we don't apply this final marginalisation step because we need the full  $p(\ufrak(0), \ufrak(1) \mid \rfrak(t_{0:N}), u_0)$ to solve for many targets.

\subsubsection{Many Target Points}
\label{section-many-targets}
With \Cref{algorithm-two-target-points} in place, solving for multiple target points $s_{0:M}$ reduces to calling it repeatedly, almost like we would do for non-probabilistic solvers.
Let $\Ical_{[a, b]}$ be the restriction of a set to the interval $[a, b]$; for example, $\Ical_{[2, 4]}(\{1, 2, 3, 4, 5\}) = \{2, 3, 4\}$.
We use $\Ical$ to split the compute grid into chunks corresponding to the targets,
\begin{align}
t_{0:N}
=
\bigcup_{n=0}^{N-1}\Ical_{[0, s_{n+1}]}(t_{0:N}).
\end{align}
Combining this separation with the Markov property 
leads to a sequential factorisation of the ODE solution over many target points instead of two, which reads
\begin{align}
\label{equation-sequential-factorisation-many-targets}
\begin{split}
&p(\ufrak(s_{0:M}) \mid \ufrak(t_{0:N}), u_0)
\\
&\quad\quad\quad
= p(\ufrak(s_M) \mid \ufrak(t_{0:N}), u_0)
\prod_{m=0}^{M-1} \xi_m.
\end{split}
\end{align}
Recall $t_N = s_M$.
The temporary variables 
\begin{align}
\xi_m \coloneqq p(\ufrak(s_{m}) \mid \ufrak(s_{m+1}), \rfrak( \Ical_{[s_0, s_m]}(t_{0:N}) ), u_0)
\end{align}
serve notational simplicity and won't be needed after \Cref{equation-sequential-factorisation-many-targets}.
Sequential factorisation is important because it reduces a difficult problem (many target points) into a sequence of simpler problems (two target points).
As a result, we solve the many-target problem with repeated calls to \Cref{algorithm-two-target-points} (recall it uses $a$ and $b$): 
\begin{enumerate}
\item Set $a=t=s_0$ and $p(\ufrak(a))=p(\ufrak(s_0) \mid u_0, \rfrak(s_0))$. Initialise $p(\ufrak(a) \mid \ufrak(t)) = \delta(\ufrak(t))$ but implement this as a Gaussian with zero covariance to allocate the correct memory. Choose a $\Delta t$. 
\item 
Assign $b=s_1$ and execute \Cref{algorithm-two-target-points}. 
Save  \mbox{$p(\ufrak(s_0) \mid \ufrak(s_1), \Ical_{[s_0, s_1]}(t_{0:N}), u_0)$}.
Set $a=s_1$ and $b=s_2$; call \Cref{algorithm-two-target-points} with the other outputs.
\item Repeat until $s_{0:M}$ has been exhausted. Additionally store $p(\ufrak(s_M) \mid \rfrak(t_{0:N}), u_0)$ and discard everything not yet stored. 
This gives \Cref{equation-sequential-factorisation-many-targets}.
\item 
Use \Cref{equation-sequential-factorisation-many-targets} to compute marginals or sample from the ODE solution \citep{kramer2024implementing} in linear time or use it for parameter estimation \citep{tronarp2022fenrir}.
\end{enumerate}
\begin{theorem}
\label{theorem-complexity-estimate}
The above procedure executes adaptive target simulation in $\Ocal(N+M)$ runtime and $\Ocal(M)$ memory. The returned values match those we would obtain through adaptive simulation and interpolation.
\end{theorem}
\begin{proof}
The $\Ocal(M)$ memory holds by construction because \Cref{algorithm-two-target-points} has $\Ocal(1)$ memory, and we call it precisely $M$ times.
The $\Ocal(N+M)$ runtime holds because every point in $t_{0:N}$ and $s_{0:M}$ is visited once.

The match with adaptive simulation is true because the only difference to adaptive simulation codes is the ``merging'' of the conditionals step at every stage of each loop in \Cref{algorithm-two-target-points}.
Tracking $p(\ufrak(t))$ and $p(\ufrak(a))$ separately is essential for adaptive target simulation not affecting the error control behaviour.
\end{proof}

\section{Related Work}
\label{section-related-work}

Our primary objective is to reduce the memory footprint of adaptive probabilistic solvers. Thus, the works by \citet{bosch2021calibrated} and \citet{schober2019probabilistic} are most closely related because they introduce the current formulation of the adaptive algorithm.
Unlike their technique, ours has constant memory requirements.

Our effort follows recent improvements to implementing Gaussian-process-based solvers \citep{kramer2024stable,kramer2022million,bosch2024parallel,kramer2021linear} by reformulating existing algorithms in a way that makes them more efficient or robust but without affecting their return values.
We are the first to target memory efficiency and extend \citeauthor{kramer2024numerically}'s \citeyearpar{kramer2024numerically} fixed-point smoother to get there (\Cref{section-two-targets}).

The articles by \citet{bosch2022pick} and \citet{bosch2024probabilistic}, as well as \citeauthor{kramer2022pnmol}'s \citeyearpar{kramer2022pnmol} partial-differential-equation algorithm are related as work on (broadly) ODE methods, but different because they build new solvers instead of making existing ones more efficient as our work does.
Studies like those by \citet{tronarp2022fenrir,kersting2020differentiable,beck2024diffusion,wu2024data} also relate through the ODE-solver connection.
However, those three works focus on applications, which is a different objective than ours.
That said, our method can be used for all of the above.

Adaptive target simulation loosely relates to ``checkpointing'' for reverse-mode derivatives of ordinary differential equations \citep[e.g.][]{griewank2000algorithm} in the sense that both approaches attempt to reduce the memory footprint of a forward-backwards pass where the backward pass depends on all intermediate values of the forward pass. 
However, beyond this similarity, the relation of our work to checkpointing is minimal.

\section{Experiments}
\label{section-experiments}

\Cref{theorem-complexity-estimate} proves that \Cref{algorithm-two-target-points} solves the problem stated in \Cref{section-problem-statement}, by implementing adaptive probabilistic ODE solvers in constant memory. 
Next, a series of experiments demonstrates how these promises translate to better performance in memory and runtime.

\paragraph{Hardware \& software setup}
All experiments run on the CPU of a consumer-level laptop with $8$ GB of RAM.
We use a small machine to demonstrate the memory efficiency of our tools:
For example, \Cref{section-experiment-overcoming-memory-limitations} uses this 8 GB of RAM to run a simulation that would require many hundreds of GB of memory with existing methods.  
Everything is implemented in JAX \citep{jax2018github} and open-sourced under this\footnote{\LibraryLink{}} link.
The simulations use double precision instead of JAX's single precision default because tolerances of $10^{-8}$ and smaller are common.
Runge--Kutta codes are imported from Diffrax \citep{kidger2021on}, not from SciPy \citep{virtanen2020scipy} because prior benchmarks have demonstrated how JAX implementations of probabilistic and non-probabilistic solvers are usually faster than SciPy \citep{kramer2024implementing}, likely due to JIT-compilation.
The latter plays a part in our experiments, too.

\subsection{Memory (and Runtime) vs. Accuracy}
\label{section-experiment-rigid-body}

\subsubsection{Motivation}
\Cref{section-method} proved drastic memory savings, and the introduction announced substantial runtime improvements. 
But do both manifest in actual simulations?
The first experiment illustrates what users can expect by switching to our fixed-point-smoother-based code.
More specifically, we choose a popular ODE benchmark and compare the memory and runtime of adaptive target simulation with that of the combination of adaptive simulation and interpolation.
The results of such a demonstration indicate which gains to expect from our proposals.
\begin{figure*}[t]
    \begin{center}
        \includegraphics{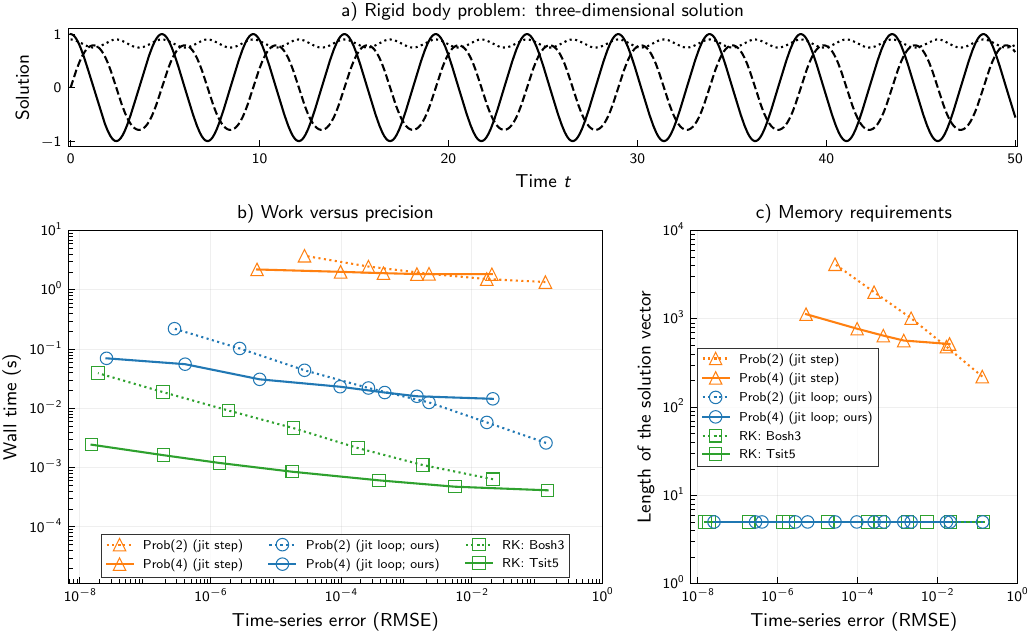}
    \end{center}
    \caption{%
    \begin{sansmath}
        \textit{\nameref*{section-experiment-rigid-body}:}
        Solution of the rigid-body problem (a).
        Wall time (b) and memory footprint (c) versus root-mean-square error for probabilistic solvers via adaptive target simulation (orange triangles) and adaptive simulation (blue circles), as well as Runge--Kutta methods (green squares).
        The probabilistic solvers use $L=2$ and $L=4$ derivatives, and Diffrax provides the Runge--Kutta methods.
        We choose ``Bosh3'' and ``Tsit5'' since their error decays similarly to that of the selected probabilistic solvers.
    \end{sansmath}
    }
    \label{figure-experiment-rigid-body}
\end{figure*}

\subsubsection{Setup}
We choose a simple but popular ODE problem,
the rigid body problem \citep[p. 244]{hairer1993solving},
which describes the rotation of a rigid body in three-dimensional, principal, orthogonal coordinates.
We chose it because existing ODE solver benchmarks do \citep{bosch2024parallel,rackauckas2017differential,rackauckas2019confederated}.
This first experiment measures the memory demands, runtime, and accumulated root-mean-square error over $M=5$ uniformly spaced points for different tolerances.
The choice of $M$ does not affect the results as long as it exceeds $M=1$ and the target points cover the time domain evenly. 
Three codes are evaluated:
\begin{enumerate}
\item Probabilistic solver: adaptive target simulation (ours; end-to-end JIT'able, constant memory)
\item Runge--Kutta methods: adaptive target simulation (end-to-end JIT'able, constant memory)
\item Probabilistic solver: adaptive simulation (can only JIT single steps due to unpredictable memory)
\end{enumerate}
\Cref{appendix-section-setup-rigid-body} lists the full setup.
In this benchmark, we want fast solvers with constantly low memory requirements. 
The role of the Runge--Kutta methods is to provide context on the runtimes and memory requirements.
We don't expect our solvers to be faster than them because we propagate (Cholesky factors of) covariance matrices, and this built-in uncertainty quantification is not free; refer to \citet{kramer2024implementing,kramer2024stable,bosch2021calibrated,bosch2022pick,bosch2024probabilistic} for more. 
Still, we hope their performance will be similar.

\subsubsection{Analysis}
The results in \Cref{figure-experiment-rigid-body} express how, as expected, the memory demands of adaptive target simulation are constant.
Solvers that rely on adaptive simulation instead of adaptive target simulation store thousands of grid points for high-accuracy solutions.
Further, the ability to use just-in-time-compilation for adaptive target simulation substantially affects runtime: the probabilistic solvers' runtimes are closer to their Runge--Kutta relatives than to their implementations without adaptive target simulation.  
For example, the $L=4$ solver via adaptive target simulation needs less than 0.1 seconds to reach accuracy $\approx 10^{-8}$, whereas not using adaptive target simulation takes multiple seconds to reach error $10^{-1}$.
The remaining gap between probabilistic algorithms and Runge--Kutta methods is due to the costs of built-in uncertainty quantification, as anticipated above.
In summary, this experiment contrasts the efficiencies of adaptive target simulation and adaptive simulation.

\subsection{Overcoming Memory Limitations}
\label{section-experiment-overcoming-memory-limitations}
\begin{figure*}[t]
    \begin{center}
        \includegraphics{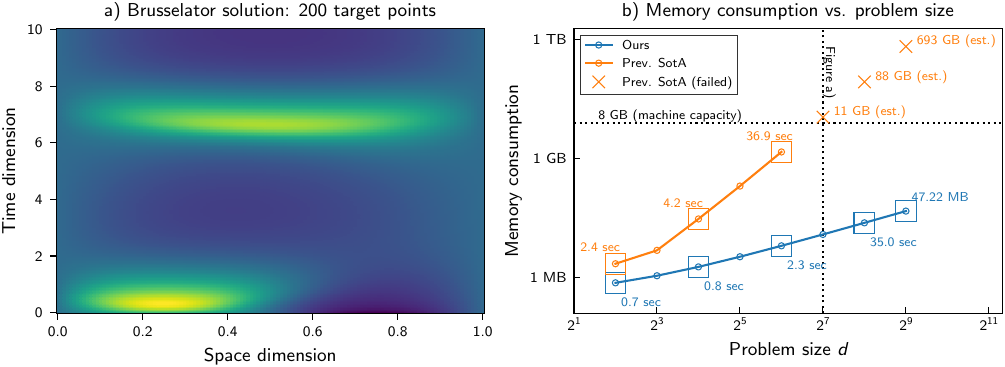}
    \end{center}
    \caption{%
        \begin{sansmath}
        \textit{\nameref*{section-experiment-overcoming-memory-limitations}:}
        Left: 
        Solution of the Brusselator problem.
        Right:
        The memory demands of the previous state-of-the-art (adaptive simulation) exceed the machine capacity somewhere between $d=2^6=64$ and $d=2^7=128$.
        Adaptive target simulation on $200$ points is far more memory efficient: at $d=512$, it consumes $47$ MB, in contrast to adaptive simulation's (estimated) $693$ GB.
        Runtime is never a bottleneck.
        \end{sansmath}
    }
    \label{figure-experiment-overcoming-memory-limitations}
\end{figure*}

\subsubsection{Motivation}
The experiment in \Cref{section-experiment-rigid-body} shows how adaptive target simulation is more memory efficient than adaptive simulation in combination with interpolation.
Next, we show that this efficiency unlocks large-scale simulations through substantial memory savings.
We solve a high-dimensional differential equation with high accuracy and investigate if our new implementation pushes the boundary of feasibility to new heights.
The results will provide intuition for which memory-challenged ODE problems were previously out of reach but can now be addressed probabilistically.

\subsubsection{Setup}
Not all ODE problems are limited by memory constraints, but many high-dimensional systems over long time intervals are.
As an example of such a system, we solve the Brusselator problem \citep{prigogine1968symmetry}, a time-dependent partial differential equation.
We discretise spatial derivatives with centred differences on an increasing number of points $d$.
The larger $d$,  the higher-dimensional the ODE, and ultimately, the harder the Brusselator becomes to solve \citep{wanner1996solving}.
We compare adaptive target simulation and adaptive simulation for increasing $d=\{2, 4, 8, ..., 512\}$, always using $200$ target points, tolerance $10^{-8}$, $L=4$ derivatives, and zeroth-order linearisation (even though the problem is stiff; why? See \Cref{appendix-section-setup-brusselator}).
We already know from \Cref{section-method,section-experiment-rigid-body} that adaptive target simulation is more memory efficient than adaptive simulation; this experiment investigates to which extent previously impossible simulations are now feasible.
Recall that all benchmarks use the CPU of a consumer-level laptop with $8$ GB of RAM.

\subsubsection{Analysis}
The results in \Cref{figure-experiment-overcoming-memory-limitations} show two phenomena: (i) runtime is never a bottleneck, not even on the finest resolutions, but memory is; and (ii) memory constraints limit adaptive simulation to resolutions below $d=64$ points in the space dimension, whereas adaptive target simulation can go up to $512$ points and likely higher. 
Concretely, for resolution $d=512$, the adaptive simulation would require $693$ GB of RAM, which is enormous for a single ODE simulation. Our code takes up less than 50 MB without sacrificing accuracy or runtime.
This contrast shows how a simulation that was previously impossible, even on large machines, can now be run in seconds on a small laptop.
Solving the Brusselator accurately is no longer out of reach.

\subsection{Further Experiments}
Our paper's primary objective is to reduce the memory footprint of probabilistic solvers; its secondary objective is efficient JAX code.
We are less interested in runtime differences, mainly because our contribution should not affect runtime beyond enabling JIT compilation.
Still, in the absence of memory limitations, users might care more about runtime than about memory, and there exist probabilistic numerics libraries that don't use JAX \citep{wenger2021probnum,bosch2024probnumdiffeq}.
\Cref{appendix-section-additional-experiments} contains more studies focusing on runtime to give insight into what happens in these two scenarios.
\begin{itemize}
    \item
    \Cref{appendix-section-experiment-pleiades} repeats the experiment in \Cref{section-experiment-rigid-body} for a more challenging ODE, the Pleiades problem.
    \item 
    \Cref{appendix-section-experiment-runtime-without-jit} discusses the runtime if we're not memory- or JIT-compilation-constrained, which suggests using our method in NumPy- or Julia- based libraries \citep{wenger2021probnum,bosch2024probnumdiffeq}.
\end{itemize}
Since memory is our main objective, these two experiments are relegated to the appendices.
\Cref{section-experiment-overcoming-memory-limitations,section-experiment-rigid-body} have already shown the superior memory requirements of adaptive target simulation.

\section{Discussion}

\subsection{Limitations}
Even though our method provably advances the memory efficiency of probabilistic solvers, it should be enjoyed with two caveats.
First, without memory limitations \emph{and} outside of JAX, ODEs can (and likely should) still be solved by combining adaptive simulation and interpolation, not with our method; \Cref{appendix-section-experiment-runtime-without-jit} discusses the nuances of this recommendation.
Second, even though our method technically applies to learning ODEs via the methods by  \citet{tronarp2022fenrir,beck2024diffusion,wu2024data}, marginal-likelihood optimisation with adaptive steps is currently unexplored.
However, we now have the tools to approach this research for decently sized problems.

\subsection{Conclusion}
This work explained how to implement the first adaptive probabilistic solver whose memory demands are fixed and chosen by the user before the simulation.
To get there, we combine Markovian priors (which are typical for probabilistic ODE solvers) with adaptive step sizes and the linear algebra of fixed-point smoothing.
The experiments have demonstrated how the memory demands of a Brusselator-simulation reduce from almost 700 GB to less than 50 MB without affecting the accuracy (\Cref{section-experiment-overcoming-memory-limitations}) and how JAX-based implementations gain orders of magnitude of runtime improvements (\Cref{section-experiment-rigid-body}).
For the future, this means that the utility of probabilistic ODE solvers once again continues to approach that of their non-probabilistic counterparts, thereby paving the way for probabilistic numerical methods in scientific machine learning.

\section*{Acknowledgements}

This work was supported by a research grant (42062) from VILLUM FONDEN. The work was partly funded by the Novo Nordisk Foundation through the Center for Basic Machine Learning Research in Life Science (NNF20OC0062606). This project received funding from the European Research Council (ERC) under the European Union’s Horizon programme (grant agreement 101125993).

\bibliography{references}

\appendix

\section*{Overview}
The following supplementary materials provide additional context for the results in the main paper.
This includes additional information on experiment setups (\Cref{appendix-setup-figure-van-der-pol-steps,appendix-section-setup-brusselator,appendix-section-setup-rigid-body}), implementation details for our method (\Cref{appendix-section-robust-marginalisation}), and two more experiments (\Cref{appendix-section-experiment-pleiades,appendix-section-experiment-runtime-without-jit}).
Like in the main paper, all experiments use the CPU of a laptop with 8 GB of RAM, implement code in JAX \citep{jax2018github}, and use double precision because tolerances below the floating-point accuracy of single precision are typical.

\section{Complete Setup: Figure 1}
\label{appendix-setup-figure-van-der-pol-steps}

As a differential equation, we choose \citeauthor{van1920theory}'s \citeyearpar{van1920theory} system, using \citeauthor{wanner1996solving}'s \citeyearpar{wanner1996solving} parametrisation, $u''(t) = 10^3 (u'(1- u^2) - u)$, $u_0 = 2, u_0' = 0$, from $t=0$ to $t=6.3$. This parametrisation is common for benchmarks \citep[e.g.,][]{bosch2021calibrated,kramer2024implementing}. 

To assemble a solver, we use $L=4$ derivatives, first-order linearisation (see \citet{tronarp2019probabilistic} for zeroth- versus first-order linearisation), a time-varying output-scale (calibrated during the forward pass; \citet{schober2019probabilistic,bosch2021calibrated}) and solve the ODE as is without transforming it into an ODE that only involves first derivatives \citep{bosch2022pick}. 
We don't factorise covariances \citep{kramer2022million} because the problem is scalar-valued.
Taylor-mode differentiation yields the initial state \citep{kramer2024stable}.

For time-stepping, we use tolerance $10^{-3}$ (both absolute and relative) in combination with proportional-integral-control \citep{gustafsson1988pi} and \citeauthor{schober2019probabilistic}'s \citeyearpar{schober2019probabilistic} error estimate.
\Cref{figure-van-der-pol-steps} shows step sizes from adaptive and fixed grids.
    The adaptive grid emerges from adaptive simulation, which means storing all intermediate steps of the forward pass (as opposed to adaptive target simulation, an implementation of which this paper contributes for probabilistic solvers). 
    This adaptive grid and the corresponding ODE solution are in \Cref{figure-van-der-pol-steps}.
    Here is how we generate the labels:

Assemble two non-adaptive grids by matching the number of steps and the smallest step size, respectively. 
Then, run a fixed-step solver on all three grids and measure the wall times. 
We use the same fixed-step solver for all grids to make the times comparable by factoring out JIT compilation as much as possible; see also \Cref{appendix-section-experiment-runtime-without-jit}.
The runtimes yield the labels in \Cref{figure-van-der-pol-steps}.

The results are plotted and discussed in \Cref{figure-van-der-pol-steps} in the main paper.

\section{Numerically Robust Marginalisation of Conditionals}
\label{appendix-section-robust-marginalisation}

For two Gaussian conditionals $\rho(x, y) = N(x; A y + a, B)$ and $\rho(y, z) = N(y; Cz + c, D)$, the marginalised conditional $\rho(x, z)$ from \Cref{equation-merge-conditionals} is
\begin{subequations}
\begin{align}
\rho(x, z) 
&= \int\rho(x, y) \rho(y, z) \diff y
\\
&= N(x; A C z + A c + a, ADA^\top + B)
\end{align} 
\end{subequations}
which results from the rules of manipulating Gaussian random variables.
To ensure numerical robustness, we must be careful with $A D A^\top + B$ to ensure symmetry and positive definiteness of all covariance matrices.

To this end, let $\sqrt{D}$ and $\sqrt{B}$ be the Cholesky factors of $D$ and $B$.
Then, we can compute the Cholesky factor of $A D A^\top + B$ only from $\sqrt{D}$ and $\sqrt{B}$, without ever assembling $D$ or $B$:
\begin{subequations}
\begin{align}
\text{(Thin-)QR-decompose}&
&
&QR 
= 
\begin{pmatrix}
\sqrt{D}A^\top
\\
\sqrt{B}^\top
\end{pmatrix},
\\
\text{then set}&
&
&\sqrt{A D A^\top + B}
= R^\top.
\end{align}
\end{subequations}
To see that this is a valid assignment, observe that $R^\top$ is lower triangular, and
\begin{subequations}
\begin{align}
A D A^\top + B
&=
\begin{pmatrix}
A \sqrt{D}
&
\sqrt{B}
\end{pmatrix}
\begin{pmatrix}
\sqrt{D}A^\top
\\
\sqrt{B}^\top
\end{pmatrix}
\\
&=
R^\top Q^\top Q R
\\
&= R^\top R.
\end{align}
\end{subequations}
As a result, $R^\top$ must be the Cholesky factor of $A D A^\top + B$.
The same QR decomposition is also used during the forward pass of the probabilistic solver; see \citet{kramer2024implementing,kramer2024stable}, and the linear algebra of \citeauthor{kramer2024numerically}'s \citeyearpar{kramer2024numerically} numerically robust fixed-point smoother.

\section{Complete Setup: Figure 2}
\label{appendix-section-setup-rigid-body}

As a benchmark ODE, we use the rigid-body problem \citep[page 244]{hairer1993solving}
\begin{subequations}
\begin{align}
\oderiv{} u_1(t)
&= -2 u_2(t) u_3(t), 
\\
\oderiv{} 
u_2(t) 
&= \frac{5}{4} u_1(t) u_3(t), 
\\
\oderiv{} u_3(t) 
&= -\frac{1}{2} u_1(t) u_2(t)
\end{align}
\end{subequations}
from $t=0$ to $t=50$, with the initial values $u_1(0) = 1$, $u_2(0) = 0$, and $u_3(0) = 0.9$.
We compute a solution with SciPy's implementation of LSODA \citep{petzold1983automatic} using tolerance $10^{-13}$ and plot it in \Cref{figure-experiment-rigid-body}.

For all solvers, we choose relative tolerances 
\begin{align}
\big\{10^{-3}, 10^{-4}, ..., 10^{-9}, 10^{-10}\big\}, 
\end{align}
discarding the high-precision third of this list for all solvers that don't use adaptive target simulation (because that would take too long to execute). 
The absolute tolerance is always $1,000\times$ smaller than the relative one.

The target grid consists of five equispaced points in the time domain.

For probabilistic solvers, we use zeroth-order linearisation, proportional-integral control, isotropic covariance factorisation, and a time-varying output scale. We initialise with Taylor-mode differentiation. We choose $L = 2$ and $L=4$.
We contrast adaptive target simulation with the combination of adaptive simulation and interpolation.
Adaptive target simulation can JIT-compile the full forward pass, but adaptive simulation can only compile a single solver step because the number of steps is unknown in advance.

For Runge--Kutta methods, we choose Bosh3 \citep{bogacki19893} and Tsit5 \citep{tsitouras2011runge} as offered by Diffrax \citep{kidger2021on} because they roughly match the error decay rate of the probabilistic solvers.
A high-accuracy reference solution comes from Dopri8 \citep{prince1981high} with tolerance $10^{-15}$.

As an error metric, we use the root-mean-square error. We always choose the best of three independent runs for the wall-time information because this choice estimates the method's efficiency without ``machine background noise''. Recall that all experiments run on the CPU of a small laptop.
The solution vector is the number of time steps in the computed solution. By construction, it is constant for adaptive target simulation. 

The results are plotted and discussed in \Cref{figure-experiment-rigid-body} in the main paper.

\section{Complete Setup: Figure 3}
\label{appendix-section-setup-brusselator}

As a test problem, we solve the Brusselator problem \citep{prigogine1968symmetry}
\begin{subequations}
\begin{align}
u(t, x) &= 1 + u(t,x)^2 v(t, x) - 4u + \frac{1}{50} \Delta u(t, x)
\\
v(t, x) &= 3u(t, x) - u(t,x)^2 v(t, x) + \frac{1}{50} \Delta v(t, x)
\end{align}
\end{subequations}
from $t=0$ to $t=10$, with $x \in [0, 1]$, and subject to the initial conditions
\begin{align}
u(0, x) =1+\sin(2\pi x)
,
\quad
v(0, x) = 3.
\end{align}
We discretise the Laplacian with centred differences on $d$ points, and vary $d \in \{2^1, 2^2, ..., 2^9\}$.
The larger $d$, the higher-dimensional the ODE, which is why, even for adaptive target simulation, the memory demands are not constant.
But also, the larger $d$, the stiffer the equation \citep{wanner1996solving}.

For all simulations, we use $L=4$ derivatives, tolerance $10^{-8}$, proportional-integral control \citep{gustafsson1988pi}, and a time-varying output scale \citep{schober2019probabilistic,bosch2021calibrated}.
As always, we initialise with Taylor-mode differentiation \citep{kramer2024stable}.

We use zeroth-order linearisation \citep{tronarp2019probabilistic} with an isotropic covariance factorisation \citep{kramer2022million} for all solvers.
Using zeroth-order linearisation may be unexpected because the Brusselator is stiff, and zeroth-order linearisation is not recommended for stiff systems (first-order linearisation should be used instead; \citet{bosch2024probabilistic}).
However, the Brusselator is also high-dimensional, and in their current implementation, first-order methods cost cubically in the ODE dimension, whereas zeroth-order methods cost linearly \citep{kramer2022million}.
Put differently, we must choose between doing too many steps (by selecting a factorised zeroth-order method) and cubic-complexity linear algebra for a high-dimensional ODE.
In practice, we found the former to be significantly faster and more memory-efficient, which is why we chose it.

We compare two modes of solvers: (i) adaptive target simulation using 200 equispaced points and (ii) adaptive simulation.
In order to not ask the machine for more memory than it can provide, we estimate the memory demands of adaptive simulation as follows before running any code: (i) simulate the terminal value of the ODE, which costs $\Ocal(1)$ memory, and track the number of steps the solver would take. Then, (ii) multiply this number of steps by the memory demands on the initialised solver state to gauge the total memory required.
If this total memory is less than 4 GB (we leave some space for background processes like code editors), run the simulation with both solvers.
If it is more than 4 GB, only run adaptive target simulation but store the predicted memory demands of adaptive simulation regardless.
We measure the wall time of every such simulation.

The results are plotted and discussed in \Cref{figure-experiment-overcoming-memory-limitations} in the main paper.

\section{Additional Experiments}
\label{appendix-section-additional-experiments}
\subsection{Runtime vs. Accuracy: Pleiades}
\label{appendix-section-experiment-pleiades}
\begin{figure*}[t]
    \begin{center}
        \includegraphics{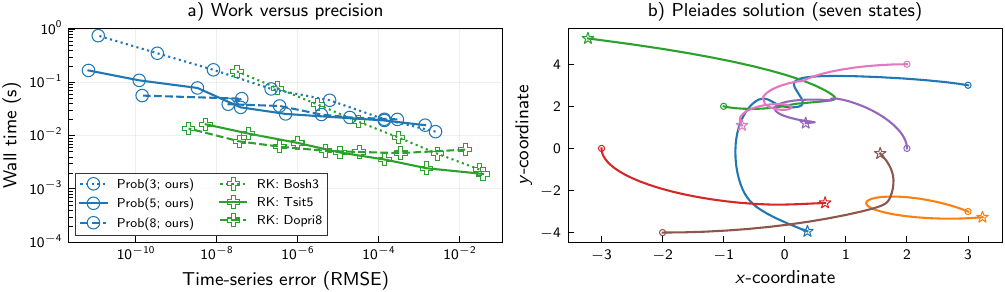}
    \end{center}
    \caption{%
    \begin{sansmath}
        \textit{\nameref*{appendix-section-experiment-pleiades}:}
        Work-precision in a); ODE solution in b).
        The probabilistic solvers (blue) are almost as fast as their non-probabilistic counterparts (green).
        Context: We expect them to be slightly slower since they use covariance matrices (this built-in uncertainty quantification is not free).
        \end{sansmath}
    }
    \label{appendix-figure-experiment-pleiades}
\end{figure*}

\subsubsection{Motivation}
The rigid-body experiment in \Cref{section-experiment-rigid-body} demonstrates how the combination of memory improvements and JIT-compilation accelerate probabilistic ODE solvers and bring their efficiency close to their non-probabilistic counterparts.
This following experiment repeats the setup from \Cref{section-experiment-rigid-body} but uses another differential equation.
Concretely, it investigates whether the performance-similarity between probabilistic and non-probabilistic adaptive target simulation carries over to more challenging ODE problems.
We consider memory questions to be answered by \Cref{section-experiment-rigid-body,section-experiment-overcoming-memory-limitations} and only focus on runtime now.
Like in \Cref{section-experiment-rigid-body}, we don't expect probabilistic solvers to be faster than Runge--Kutta methods in terms of work per attained precision, but we hope their speed is similar.

\subsubsection{Setup}
In this experiment, we solve the Pleiades problem from celestial mechanics \citep{hairer1993solving},
\begin{align}
\label{equation-pleiades}
u''(t) = f(u(t)), \quad u(0) = u_0,
\end{align}
with $f$ and $u_0$ as on page 245 in \citeauthor{hairer1993solving}'s \citeyearpar{hairer1993solving} book.
The Pleiades problem describes the motion of seven stars, formulated as a 14-dimensional differential equation involving second time-derivatives.
When selecting solvers, we mimic the setup from \Cref{section-experiment-rigid-body} but drop adaptive simulation because we now know it will be far less efficient.
Instead, we add two more solvers in adaptive-target-simulation mode.
Concretely:

We solve \Cref{equation-pleiades} from $t=0$ to $t=3$. \Cref{equation-pleiades} involves second-time derivatives, and whereas probabilistic methods can solve these types of problems natively \citep{bosch2022pick}, Diffrax does not offer Runge--Kutta methods that can do the same. Therefore, the Runge--Kutta methods transform the problem into one involving only first time-derivatives before simulating.

For Runge--Kutta methods, we vary the relative tolerances from $\{10^{-3}, ..., 10^{-10}\}$. For the probabilistic solver, we vary from $10^{-2}$ to $10^{-9}$ because we found this leads to similar accuracies. 
For work-precision diagrams, the tolerance need not match between methods as long as the resulting work-precision ratios are comparable.
The absolute tolerances are always $10^3\times$ smaller than the relative tolerances, which is typical; for example, it's the default in Scipy's ODE-solver \citep{virtanen2020scipy}. 

We save the solution at $50$ equispaced points in the time domain.
We consider the wall time (in seconds, best of three runs to remove ``background machine noise'' as much as possible) and the root-mean-square error.

Probabilistic solvers always use zeroth-order linearisation \citep{tronarp2019probabilistic}, an isotropic covariance factorisation \citep{kramer2022million}, time-varying output scales \citep{schober2019probabilistic}, and initialise via Taylor-mode differentiation \citep{kramer2024stable}. 

We choose $L \in \{3, 5, 8\}$ derivatives. Runge--Kutta methods are chosen to match the probabilistic solvers' error decay rate; like before, we include Bosh3 \citep{bogacki19893}, Tsit5 \citep{tsitouras2011runge}, and Dopri8 \citep{prince1981high} via Diffrax \citep{kidger2021on}.

We compute a reference solution with Dopri5 \citep{shampine1986some} and tolerance $10^{-15}$, to evaluate approximation errors and for visualisation.

\subsubsection{Analysis}
The results in \Cref{appendix-figure-experiment-pleiades} show that, once again, the probabilistic solvers are similarly efficient to Runge--Kutta methods but slightly slower.
In comparison to prior work \citep[e.g.][]{kramer2024stable}, being only slightly slower is a success. Further, the results match \citeauthor{kramer2024implementing}'s \citeyearpar{kramer2024implementing} results for terminal-value simulation, which is reassuring.
In general, this experiment underlines the value of adaptive target simulation for computing time-series benchmarks for probabilistic solvers.

\subsection{Runtime on a Small Problem and Without JIT-Limitations}
\label{appendix-section-experiment-runtime-without-jit}

\begin{table*}[t]
\caption{Runtime on the three-body problem:
Adaptive simulation (AS; but by ``cheating'' JIT-compilation, see the explanation above) versus adaptive target simulation (ATS; our method). Lower is better. Winners in bold.}
\label{table-results-sampling}
\begin{center}

\begin{tabular}{c c c c c}
\toprule
No. Samples & Tolerance & No. steps & Time (s): {AS} & Time (s): {ATS (ours)} \\
\midrule
 5 & $10^{-4}$ & 448 & \winner 0.007 & 0.015 \\
 5 & $10^{-7}$ & 2,570 & \winner 0.037 & 0.065 \\
 5 & $10^{-10}$ & 14,469 & \winner 0.215 & 0.347 \\
\midrule
 50 & $10^{-4}$ & 448 & \winner 0.014 & 0.015 \\
 50 & $10^{-7}$ & 2,570 & 0.098 & \winner 0.076 \\
 50 & $10^{-10}$ & 14,469 & 0.536 & \winner 0.374 \\
\midrule
 500 & $10^{-4}$ & 448 & 0.127 & \winner 0.032 \\
 500 & $10^{-7}$ & 2,570 & 0.700 & \winner 0.080 \\
 500 & $10^{-10}$ & 14,469 & 3.851 & \winner 0.393 \\
\bottomrule
\end{tabular}

\end{center}
\end{table*}
\subsubsection{Motivation}
The constant memory requirements of adaptive target simulation have been studied by \Cref{section-experiment-overcoming-memory-limitations,section-experiment-rigid-body}.
The conclusions from these experiments do not depend on the programming environment.
However, the runtime (as in, wall time) gains from \Cref{figure-experiment-overcoming-memory-limitations,figure-experiment-rigid-body} were specific to JAX because JAX cannot compile a function whose memory demands are unknown at compilation time.
This experiment eliminates JAX-specific considerations from the wall-time comparison.
In other words, ``what if JAX could compile adaptive simulation despite its unpredictable memory requirements?''
The results of this experiment are instructive for probabilistic solver codes that are implemented in NumPy \citep{harris2020array} or Julia \citep{bezanson2017julia}; namely, those by  \citet{bosch2024probnumdiffeq,wenger2021probnum}.
We investigate in which case they benefit from our implementation of adaptive target simulation \emph{beyond} memory limitations, which is a scenario outside our primary objectives (memory-challenging equations or JAX).

\subsubsection{Setup}
This experiment investigates the solvers' performances without JAX's compilation constraints.
However, we would still like to use our JAX code because we don't want to reimplement ODE solvers for this one experiment.
Using JAX code for a JAX-independent benchmark necessitates cheating a little in the setup.
However, we only cheat to make competing algorithms more efficient than they would be otherwise, and we leave our implementation as is.
Concretely:

All methods solve the restricted three-body problem \citep[page 129]{hairer1993solving}, another ODE involving second derivatives (after \Cref{appendix-section-experiment-pleiades}).
We replicate the exact setup from \Cref{appendix-section-experiment-pleiades} with two exceptions: we don't include Runge--Kutta methods because they are irrelevant to this benchmark, and all probabilistic solvers use $L=4$ derivatives.

Unlike the previous experiments, we don't measure work-versus-precision for ODE simulation.
Instead, to emphasise using the probabilistic ODE solution for increasingly complex tasks, we compute an increasing number of samples from the ODE posterior \citep{kramer2024implementing}.
Few samples imply ``cheap processing'' of the ODE solution, and many samples will represent ``expensive processing''.
We expect that adaptive target simulation will be more efficient than adaptive simulation for ``expensive processing'' (even with cheating the JIT compilation), and this experiment investigates just how expensive the processing needs to be for this to be true.
We compare the following two routines.

\textit{A. Adaptive target simulation (ours, no cheat)}
 \begin{enumerate}
        \item Solve adaptively using $50$ target points
        \item Compute $K$ joint samples from the posterior
    \end{enumerate}
    \textit{B. Adaptive simulation (cheat):}
    \begin{enumerate}
        \item Solve adaptively storing the full grid
        \item Augment the full grid with $50$ target points
        \item Solve using the augmented grid as a fixed grid
        \item Compute $K$ joint samples from the posterior
        \item Subselect the locations of the samples that correspond to the 50 target points.
\end{enumerate} 
The cheat in B is that the adaptive solve is excluded from the timing to enable JIT-compilation.
Method B is not a realistic setup because it uses a JIT-compiled code for something that can't be JIT-compiled on the first run. Still, it gives insights into what would happen in environments not limited by JAX's compilation restrictions, e.g. Julia \citep{bezanson2017julia}.

We are looking for fast simulation across multiple sample numbers in this experiment.

\subsubsection{Analysis}
The results in \Cref{table-results-sampling} show that for a low number of samples, the (JIT-cheated) combination of adaptive simulation and interpolation is faster than adaptive target simulation, with the roles reversed for larger sample counts.
For small sample counts, the differences are small; for large sample counts, the differences are large.
Notably, we found adaptive simulation to run into memory issues for more than 500 samples, which is why we only show up to 500 samples.
This experiment suggests two conclusions:
(i) in the absence of memory- and JIT-limitations and if we do cheap computations with the ODE solution, the combination of adaptive simulation and interpolation remains the state of the art; (ii) if either of those three constraints does not hold, adaptive target simulation should be used, independent of the programming environment.
However, recall from \Cref{section-experiment-rigid-body,section-experiment-overcoming-memory-limitations} and \Cref{appendix-section-experiment-pleiades} that in JAX, adaptive target simulation always wins.

\end{document}